\documentclass[reqno]{amsart}

\usepackage[latin1]{inputenc}
\usepackage{amsmath,amssymb}
\usepackage{amsthm}
\usepackage{graphicx}
\usepackage{comment}
\usepackage{enumitem}

\newtheorem*{theorem*}{Theorem}
\newtheorem{theorem}{Theorem}[section]
\newtheorem{proposition}[theorem]{Proposition}
\newtheorem{lemma}[theorem]{Lemma}

\theoremstyle{definition}
\newtheorem{definition}[theorem]{Definition}
\newtheorem{remark}[theorem]{Remark}
\newtheorem{example}[theorem]{Example}

\newcommand{\R}{\mathcal{R}}
\newcommand{\C}{\mathcal{C}}
\newcommand{\A}{\mathcal{A}}
\newcommand{\G}{\Gamma}
\newcommand{\w}{\mathbf{w}}
\newcommand{\e}{\varepsilon}
\renewcommand{\H}{\operatorname{Heap}}

\newcommand{\Wfc}{W^{FC}}
\newcommand{\braidr}[2]{[#1\cdots]_{#2}}
\newcommand{\braidl}[2]{[\cdots #1]_{#2}}
\newcommand{\TL}{\operatorname{TL}}
\newcommand{\Ared}{\mathcal{A}_{ch}}

\begin{document}

\title[Length of fully commutative elements]{On the length of fully commutative
elements}

\author[Philippe Nadeau]{Philippe Nadeau}
\address{CNRS, Institut Camille Jordan, Universit\'e Claude Bernard Lyon 1,
69622 Villeurbanne Cedex, France}
\email{nadeau@math.univ-lyon1.fr}
\urladdr{http://math.univ-lyon1.fr/{\textasciitilde}nadeau}

\date{\today}

\subjclass[2010]{}
\keywords{Coxeter groups, fully commutative elements, Coxeter length, rational functions, finite automata, periodic sequences, reduced expressions, Temperley--Lieb algebra}

\begin{abstract} In a Coxeter group $W$, an element is \emph{fully commutative} if any two of its reduced expressions can be linked by a series of commutation of adjacent letters. These elements  have particularly nice combinatorial properties, and also index a basis of the generalized Temperley--Lieb algebra attached to $W$. 

We give two results about the sequence counting these elements with respect to their Coxeter length. First we prove that it always satisfies a linear recurrence with constant coefficients, by showing that reduced expressions of fully commutative elements form a regular language. Then we classify those groups $W$ for which the sequence is ultimately periodic, extending a result of Stembridge. These results are applied to the growth of generalized Temperley--Lieb algebras.
\end{abstract}

\maketitle


\section{Introduction and Results}

Coxeter groups are defined by a finite set $S$ of generating involutions which are subject to {braid relations }of the form $st\cdots=ts\cdots$ where both sides have length $m_{st}\geq 2$. These integers are gathered in a  Coxeter matrix $M$ indexed by $S$. This abstract presentation hides the fact that Coxeter groups occur naturally in a number of algebraic and geometric settings. They arise frequently as generalized reflection groups, and as a matter of fact coincide with groups generated by orthogonal reflections in the finite case. The interest in Coxeter groups also comes  via {Hecke algebras}, which are one-parameter deformations of their group algebras. These Hecke algebras have natural bases indexed by elements of $W$.

Coxeter groups possess rich combinatorics, starting in the finite case which includes the symmetric groups for instance. A lot of the combinatorial structure is encoded in the {\em reduced expressions} an element $w\in W$, which are words $s_1\cdots s_l$ in $S$ representing $w$ such that $w$ can not be written with less than $l$ generators. This integer $l$ is the \emph{length }$\ell(w)$.  Properties of reduced expressions and length form arguably the core of the combinatorics of Coxeter groups, for which the standard reference is Bj\"orner and Brenti's book~\cite{BjorBrenbook}. \medskip

An element of $W$ is called {\em fully commutative} (FC) if any two of its reduced decompositions can be linked by a series of commutation relations of adjacent letters. Note that for a general element $w$, it is known that any two of its reduced decompositions are linked by braid relations (Matsumoto's property). From this point of view, the reduced expressions of a given FC element are thus better behaved. In his seminal work~\cite{St1}, Stembridge shows a number of natural characterizations of FC elements; for instance, their reduced expressions correspond bijectively to linear orderings of a partial order. In the simply laced case, \emph{i.\,e.} when $m_{st}\leq 3$ for all $s,t\in S$, Fan and Stembridge~\cite{FanStem} give further characterizations of FC elements in terms of the root system attached to $W$.

 The Temperley--Lieb algebra is a famous diagram algebra born in statistical mechanics~\cite{TemperleyLieb}, which was later recognized as a quotient of the Hecke algebra of type $A_n$~\cite{JonesAnnals}. This was later used to define generalized Temperley--Lieb algebras for all Coxeter groups as quotients of Hecke algebras, and then prove that FC elements index a basis of these algebras~\cite{Fan,Graham}. We will give applications of our results to the growth of these algebras at the end of this introduction, results that we now proceed to state.
\medskip

We investigate in this work the enumeration of FC elements in a general Coxeter group. Let us first note that there may be a finite number of FC elements even though $W$ is infinite: in fact in his work~\cite{St1}, Stembridge precisely classifies when this happens, and in a subsequent paper~\cite{St2} he enumerates the number of FC elements in all such cases. To extend these results to general Coxeter groups, it is natural to enumerate elements according to their length.

 We write $\Wfc_l$ for the set of all FC elements of length $l$, and let $\Wfc(t)=\sum_{l\geq 0} |\Wfc_l| t^l$ be the corresponding generating function. It was computed in affine type $\widetilde{A}_n$ in~\cite{HanJon}, and this was recently extended to all finite and affine Coxeter groups in~\cite{BJN}.
\smallskip

 Our first contribution deals with the nature of the generating function $\Wfc(t)$. The following result shows that it is essentially in the simplest  possible class.
 
\begin{theorem}[Rationality]
\label{th:rationality}
For any Coxeter group $W$, the series $\Wfc(t)$ is rational.
\end{theorem}

An equivalent formulation is that the sequence $|\Wfc_l|, l=0,1,\ldots$ satisfies a linear recurrence with constant coefficients for $l$ large enough.
\smallskip

There are a number of results studying this enumeration for various subsets of $W$. First and foremost, the length generating function of $W$ itself is rational. This is classically shown by induction using elementary combinatorics of Coxeter groups, see \cite[Section 7.1]{BjorBrenbook} or \cite[Section 5.12]{Humphreys}. A lesser known result is that this is true for the set of reflections of $W$, which are the conjugates of generators $S$ ~\cite{deMan}. Another (simpler) example is that elements possessing a unique reduced expression also have a rational generating function~\cite{cassweb}.
\smallskip

 The proof of Theorem~\ref{th:rationality} is not direct; unlike in the case of $W$, there does not seem to be a simple recursive decomposition of $\Wfc$ which behaves well with respect to length. Instead, we will use the theory of \emph{finite automata}. In Section~\ref{sec:rationality}, we will design an automaton recognizing a language with generating function $\Wfc(t)$, which automatically entails rationality. We start by showing that the set $\R(\Wfc)$ of reduced expressions of elements of $\Wfc$ is recognized by an explicit automaton $\Ared(W)$. The design of this automaton and proof of its correctness form the core of Section~\ref{sec:rationality}. This recalls here also a celebrated result of Brink and Howlett~\cite{BrinkHowlett} (see also~\cite[Section 4.8]{BjorBrenbook}) who show that the set of reduced expressions of {\em all} elements can be recognized by a finite automaton. To go from expressions to elements, we use the lexicographical order to select normal forms for elements of $\Wfc$; Theorem~\ref{th:rationality} follows then   thanks to a result of Anisimov and Knuth.

 The proof we sketched is completely constructive, and allowed us to implement the computation of $\Wfc(t)$ given a Coxeter group $W$.
\bigskip

We now describe our second contribution. As mentioned above, one of the main results of~\cite{St1} is the classification of $W$ such that $\Wfc$ is finite; let us say that $W$ is FC-finite in this case. Now the main contribution of~\cite{BJN} is that for any irreducible, 
affine Coxeter group $W$, the sequence $(|\Wfc_l|)_{l\geq 0}$ is ultimately 
periodic. It is then a natural question to wonder if this holds for other 
Coxeter groups, a question we answer in the following theorem. 

\begin{theorem}[Classification]
\label{th:classification_intro}
 Let $W$ be an irreducible Coxeter group. The sequence $(|\Wfc_l|)_\ell$ is ultimately periodic if and only if $W$ is either  FC-finite, an affine Coxeter group or has one of the two types
\begin{center}
 \raisebox{0.5cm}{\includegraphics[scale=0.7]{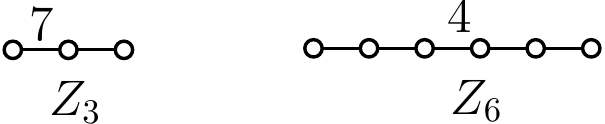}}
\end{center}
In all other cases, $(|\Wfc_l|)_\ell$ grows exponentially fast.
\end{theorem}

The Coxeter groups listed in Theorem~\ref{th:classification_intro} are given by their Coxeter graph in Figure~\ref{fig:FC_Periodic_Graphs}. We note that both $Z_3$ and $Z_6$ occur as  hyperbolic Coxeter groups, as compact and paracompact examples respectively~\cite{Humphreys}.
\smallskip

Theorem~\ref{th:classification_intro} is proved in Section~\ref{sec:classification}, where it will appear again as Theorem~\ref{th:classification}. The proof goes as follows: first, we check that all groups from the list have an ultimately periodic sequence. To prove that all other groups have exponential growth, we first identify which are the minimal ones according to a certain order on Coxeter groups which is compatible with the length of FC elements. Then for each of these minimal types, we either identify a subfamily of FC elements which grows exponentially fast, or resort to Theorem~\ref{th:rationality} and its implementation to obtain the explicit function $\Wfc(t)$.
\medskip

{\noindent\bf Growth of generalized Temperley--Lieb algebras}

For $W$ a Coxeter group with Coxeter matrix $M=(m_{st})_{s,t\in S}$ and $\mathbf{k}$ a unital commutative ring, the associated  {\em Hecke algebra} $\mathcal{H}(W)$ is the $\mathbf{k}$-algebra given by generators $T_s$ and subject to the relations 
$T_s^2=(y-1)T_s+y\mathbf{1}\quad\text{for }s\in S$, and $\braidr{T_sT_tT_s}{m_{st}} = \braidr{T_tT_sT_t}{m_{st}}\quad\text{if }m_{st}<+\infty.$

Here $y\in\mathbf{k}$; if $y=1$ we get the group algebra of $W$. For any $w\in W$, define $T_w\in \mathcal{H}(W)$ by picking any reduced decomposition $s_{1}\cdots s_{m}$ for $w$ and setting $T_w:=T_{s_{1}}\cdots T_{s_{m}}$. These elements $T_w$ then form a basis of $ \mathcal{H}(W)$ (see for instance~\cite{Humphreys}).

 The {\em generalized Temperley--Lieb algebra} $\TL(W)$ is defined as  $\mathcal{H}(W)/\mathcal{J}$ , where $\mathcal{J}$ is the ideal generated by the elements 
$\sum_{w\in W_{s,t}}T_w$ where $3\leq m_{st}<+\infty$
and $W_{s,t}$ is the finite dihedral subgroup generated by $s$ and $t$. 
\begin{theorem*}[{\cite[Theorem 6.2]{Graham}}]
Let $b_w$ be the image of $T_w$ in $\TL(W)$. The elements $b_w$ for $w\in \Wfc$ form a basis of $\TL(W)$.
\end{theorem*}

Consider now the natural filtration $\TL(W)_0\subset \TL(W)_1\subset \cdots$ of $\TL(W)$, where $\TL(W)_l$ is the linear span in $\TL(W)$ of all products $b_{s_{1}}\cdots b_{s_{k}}$ with $k\leq l$. A linear basis for $\TL(W)_l$ is given by the $b_w$ where $w \in \Wfc$ satisfies $ \ell(w)\leq l$.

Let the {\em growth} of $\TL(W)$ be $G^W:\ell\mapsto \dim \TL(W)_l=|\{w\in\Wfc\,:\,\ell(w)\leq l\}|.$ The definition depends on the choice of generators, however it makes sense to speak of an exponential growth, or a polynomial growth of degree $d$ without specifying the choice of generators. We refer to~\cite{GrowthAlgebrasBook,Ufnarovskij} for more informations on growth of algebras. 

 The following theorem is a translation of Theorem~\ref{th:classification_intro} in terms of Temperley--Lieb algebras. It generalizes to all Coxeter systems $(W,S)$ the results of ~\cite{Ukraine_simple} and  ~\cite{Ukraine_four} which dealt with the simply-laced system and the rank $4$ case respectively. The families $(PF),(PA)$ and $(PE)$ are given in Figure~\ref{fig:FC_Periodic_Graphs}.

\begin{theorem}
\label{th:growthTL}
 Let $(W,S)$ be an irreducible Coxeter system.\\
 $\bullet$ If $W$ is in $(PF)$, the algebra $\TL(W)$ is finite dimensional.\\
 $\bullet$ If $W$ is in $(PA)\cup (PE)$, the algebra $\TL(W)$  has {\em linear growth}.\\
 $\bullet$ In all other cases, the algebra $\TL(W)$ has \emph{exponential growth}.
\end{theorem}

\begin{remark} The \emph{Gelfand-Kirillov dimension} of a finitely generated algebra $A$ with growth function $G$ can be defined as $\operatorname{limsup}\left(\log G(l)/\log l\right)$ when $l$ tends to infinity, where $G$ can be computed with respect to any choice of generators.  In the three cases distinguished by Theorem~\ref{th:growthTL}, the Gelfand-Kirillov dimension of $\TL(W)$ is thus $0,1$ or $\infty$ respectively.
\end{remark}

To interpret Theorem~\ref{th:rationality}, it is slightly more pleasant to work with the {\em nil Temperley--Lieb algebra} ${\rm nTL}(W)$: this is the graded algebra associated with the above filtration of $\TL(W)$. By definition, this means that its $l$th graded component is given by the free module $\TL(W)_l/\TL(W)_{l-1}$, and the multiplication is inherited from $\TL(W)$. These algebras seem to have been studied only in type $A_{n-1}$ ~\cite{FominGreene} and $\widetilde{A}_{n-1}$~\cite{Postnikov, BenkartMeinel}. 

The $l$th graded component of ${\rm nTL}(W)$ has a basis $(u_w)$ indexed by FC element of length $\ell$, so its Hilbert series is $\Wfc(t)$, so that we obtain the following reformulation of Theorem~\ref{th:rationality}.

\begin{theorem}
\label{th:nTLrational}
For any Coxeter group $W$, the Hilbert series of ${\rm nTL}(W)$ is rational. 
\end{theorem}

This article is structured as follows. In Section~\ref{sec:coxeter_fullycom} we give necessary definitions about Coxeter groups and fully commutative elements. Section~\ref{sec:rationality} is devoted to the proof of Theorem~\ref{th:rationality}. The key result is that the reduced expressions of FC elements form a regular language, see Theorem~\ref{th:reducedfcregular}. The proof of Theorem~\ref{th:classification_intro} is the content of Section~\ref{sec:classification}, and we conclude with some perspectives in Section~\ref{sec:perspectives}.

\section{Coxeter groups and fully commutative elements}
\label{sec:coxeter_fullycom}

\subsection{Coxeter groups}
\label{sub:coxeter}
 Let $S$ be a finite set, and $M=(m_{st})_{s,t\in S}$ a symmetric matrix  satisfying $m_{ss}=1$ and $m_{st}\in \{2,3,\ldots\}\cup\{\infty\}$. The \emph{Coxeter group} $W$ associated with $M$ is generated by $S$  with relations $(st)^{m_{st}}=1$ for all $s,t\in S$ such that $m_{st}<\infty$. 
 
 Equivalently, $W$ is generated by $S$ whose elements satisfy $s^2=1$ as well as the {\em braid relations} $[st\cdots]_{m_{st}}=\braidr{ts}{m_{st}}$  if  $s\neq t$ and  $m_{st}<\infty$. We call {\em braid factors} the two alternating words in $\{s,t\}$ of length $m_{st}$. In the special case $m_{st}=2$ we have a {\em commutation relation}.

 A Coxeter group is encoded by its associated {\em Coxeter graph} $\Gamma(W)$, which has vertex set $S$, edges $\{s,t\}$ if $m_{st}>2$ with edge labels $m_{st}$ if $m_{st}>3$. The group $W$ is called {\em irreducible} if $\Gamma(W)$ is connected.
\smallskip

\begin{example}
\label{ex:coxetergraph}
  Define the group $W_0$ with generators $\{s_0,s_1,s_2\}$ and matrix $M$ given by  $m_{s_0s_2}=2$, $m_{s_0s_1}=4$ and $m_{s_1s_2}=5$. The relations of $W_0$ are $s_i^2=1$ for $i=0,1,2$ together with the braid relations $
su=us; stst=tsts; tutut=ututu$. Its graph is thus $\Gamma(W_0)=$ \raisebox{-0.25cm}{\includegraphics{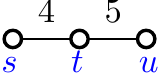}}.

\end{example}

Let $(W,S)$ be a Coxeter system. The {\em length} $\ell(w)$ of an element $w\in W$ is the smallest integer $l$ such that $w=s_1s_2\cdots s_l$ for some $s_i\in S$. Such minimal words representing $w$ are called {\em reduced expressions} for $w$, and form  together the set $\R(w)$. For any subset $A\subseteq W$, define the two generating functions
\begin{align*}
 A(t)=\sum_{w\in A}t^{\ell(w)}\quad \text{and} \quad\R(A)(t)=\sum_{w\in A}|\R(w)|t^{\ell(w)}
\end{align*} 

These notions can be illustrated with the \emph{Cayley graph} of $W$ with respect to the generators $S$: this is the graph with vertex set $W$ and edges $\{w,ws\}$. Then $\ell(w)$ is the graph distance from the identity to $w$, and reduced words correspond to paths realizing this distance. The series $A(t)$ and $\R(A)(t)$ correspond respectively to the \emph{spherical} and \emph{geodesic} growth series of the subset $A$.

We finally record a fundamental property of reduced expressions.

\begin{proposition}[Matsumoto's property~\cite{Matsumoto}]
\label{prop:matsumoto}
Given two words in $\R(w)$, one can always go from one to the other by using only braid relations.
\end{proposition}


\subsection{Fully commutative elements}
\label{sub:fullycomm}
Given a graph $\G$ with vertex set $S$, a $\G$-commutation class is an equivalence class of words in $S^*$ under the equivalence relation generated by the commutation relations $\w st\w'\sim\w ts\w'$ if $\{s,t\}$ is not an edge of $\Gamma$. We write $\C_\G(\w)$ for the commutation class of $\w$.
\begin{definition}[Fully commutative elements]
\label{defi:fullycomm}
An element $w\in W$ of a Coxeter group is fully commutative (FC in short) if $\R(w)$ is a $\G(W)$-commutation class. 
\end{definition}
The set of FC elements of $W$ will be denoted by $\Wfc$. An element is FC if one can always go from one reduced expression to another by using only commutation relations. In comparison, Proposition ~\ref{prop:matsumoto} says that if all braid relations are allowed, then all elements of $W$ have this property.
\smallskip

By a result of Stembridge \cite[Prop. 2.1]{St1}, if $w$ is FC then no word in $\R(w)$ contains a braid word. Therefore FC elements are in bijective correspondence with commutation classes $\C$ of reduced words such that no element of $\C$ contains a braid factor. It is in fact particularly simple to check that expressions are reduced in this case, as the following lemma shows.

\begin{lemma}
\label{lem:FCcommclasses}
The map $w\mapsto \R(w)$ restricts to a bijection between FC elements and commutation classes $\C$ such that no element in $\C$ has a braid factor $\braidr{st}{m_{st}}$ for a $m_{st}>2$ or a square $s^2$.
\end{lemma}

\begin{proof}
We must show that such a commutation class $\C$ is formed of reduced words. Now by Tits' result~\cite{Tits68}, any word representing $w\in W$ can be transformed into a reduced word by a series of braid relations and deletions of squares $s^2$. But  there are no braid relations to apply in any word of $\C$, so these are indeed reduced.
\end{proof}

\subsection{Heaps}
\label{sub:heaps}  
To deal nicely with commutation classes, we will use the concept of {\em heaps} as developed by Viennot~\cite{ViennotHeaps}. we use the definition given in \cite{GreenBook}.

\begin{definition}[$\G$-Heaps] Let $\Gamma$ be a finite graph with vextex set $V$ and edge set $E$. A {\em heap} on  $\G$ is a poset $(H,\preceq)$ together with labels $\e:H\to V$ such that (a) if $\{s,t\}\in E$, the set $H_{s,t}$ of elements of $H$ with labels in $\{s,t\}$ form a chain, and (b) the order $\preceq$ is the transitive closure of the union of all chains above.
\end{definition}

 More explicitly, (b) amounts to saying that $i\prec j$ in $H$ only if there exists a chain $i= h_0\prec h_1 \prec h_2\prec \cdots\prec  h_r=j$ in which $\{\e(h_p),\e(h_{p+1})\}\in E$ for any $p<r$. Since $H_{s,t}$ is linearly ordered, it can be encoded with the word in $V^*$ formed by the labels of $H_{s,t}$ in increasing order.
We shall in fact consider heaps up to isomorphism: $H$ and $H'$ are isomorphic if there exists a poset isomorphism from $H$ to $H'$ which preserves labels.

\smallskip
\begin{figure}[!ht]
\includegraphics{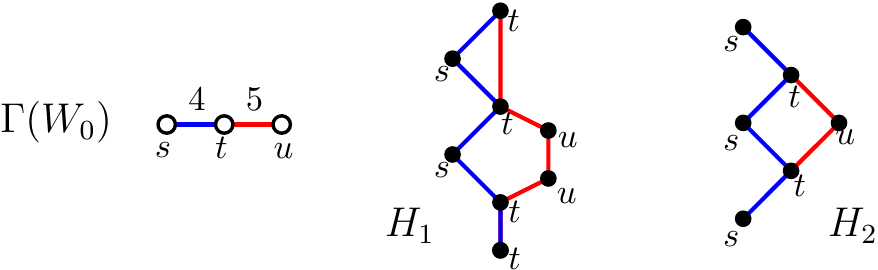}
\caption{Two heaps associated with the graph of $W_0$. The heaps are represented by their defining chains, oriented from bottom to top.\label{fig:heaps}}
\end{figure}

There is a function $\w\mapsto \H_\G(\w)$ between words and $\G$-heaps, defined as follows. Let $\w=s_{1}\cdots s_{l}\in V^*$. Define a partial ordering $\prec$ of  $\{1,\ldots, l\}$ as follows: set $i\prec j$ if $i<j$ and $\{s_{i},s_{j}\}$ is an edge of $\G$, and extend to a partial order by transitivity. We get a heap $\H_\G(\w)$  by naturally labeling $i$ by $s_{a_i}$. For instance, words $ttusutst$ and $stsuts$ give rise to heaps $H_1$ and $H_2$ of Figure~\ref{fig:heaps} respectively. 

The fundamental property is that {\em $\H_\G$ is a bijection between $\G$-commutation classes and $\G$-heaps}. The inverse bijection is obtained by considering all linear extensions of $H$ and reading the words formed by the labels.
\medskip

\subsection{FC heaps} Now assume $\G=\Gamma(W,S)$ is the Coxeter graph of $(W,S)$, so that $V=S$. By Lemma~\ref{lem:FCcommclasses}, FC elements are in bijection with certain $\G$-commutation classes; it is natural to try to characterize the corresponding heaps. Consider the following two properties of a $\G$-heap $H$:

 (h1) $H$ contains no covering relation $i\prec j$ with $\e(i)=\e(j)$.

 (h2) $H$ contains no convex chain $i_1\prec i_2\prec i_3\prec\cdots \prec i_{m_{st}}$ with alternating labels $s,t,s,\ldots$ for $m_{st}<+\infty$. Here $K\subset H$ is convex if whenever  $h\prec x\prec h'$, we have that $h,h'\in K$ imply $x\in K$.
\smallskip

\begin{proposition}[{\cite[Proposition 3.3]{St1}}]
\label{prop:bij_fc_elements_heaps}
The map $\w\mapsto \H_\G(\w)$ induces a bijection between $\Wfc$ and the set of $\G$-heaps which satisfy (h1) and (h2).
 Moreover $\ell(w)=|\H_\G(w)|$, and the reduced expressions for $w$ correspond bijectively to linear extensions of $\H_\G(w)$.
\end{proposition}

The proof uses Lemma~\ref{lem:FCcommclasses}: (h1) ensures that the linear extensions of $\H(w)$ have no occurrence of $s^2$, while (h2) prevents linear extensions of $H$ to contain a braid word as a factor.

A $\G$-heap satisfying (h1) and (h2) will be called a \emph{FC heap}.
Given a $\G$-heap $H$, denote by $\max(H)$ the subset of $S$ formed by the labels of the maximal elements of $S$. Then we have the following proposition.

\begin{proposition}
\label{prop:maxdes}
If $w\in \Wfc$ and $H=\H(w)$, then for any $s\in S$ one has $\ell(ws)>\ell(w)$ if and only if $s\in \max(H)$.
\end{proposition}

Indeed $\ell(ws)>\ell(w)$ iff $w$ has a reduced expression ending with $s$ iff $H$ has a linear extension ending with a label $s$ iff $s\in \max(H)$.
\section{Rationality}
\label{sec:rationality}
In this section we will prove Theorem~\ref{th:classification_intro}, which states that $\Wfc(t)$ is a rational series for any Coxeter system $(W,S)$. The fundamental tool we use is the concept of language recognized by a finite automaton, which we recall in the first paragraph.


\subsection{Finite automata} We refer to the book~\cite{Saka} for a general reference on the theory of automata. We only need the most basic notions.

 A \emph{finite, deterministic, complete automaton $\A$} on the alphabet $S$ is the datum $(Q,\delta,q_0,F)$ of a finite set of \emph{states} $Q$, a transition function $\delta:Q\times S\to Q$, an initial state $q_0\in Q$, and a subset $F\subseteq Q$ of final states. The function $\delta$ is extended to $\delta:Q\times S^*\to Q$ by defining recursively $\delta(q,\epsilon)=q$ and $\delta(q,\w s)=\delta(\delta(q,\w),s)$. We identify the automaton $\A$ with its associated labeled multigraph with vertex set $Q$ and oriented, labeled edges $q\stackrel{s}{\to} q'$ whenever $\delta(q,s)=q'$; see Figure~\ref{fig:petitautomate} for an example in which the final states are $q_0$ and $q_1$.
\begin{figure}[!ht]
\includegraphics{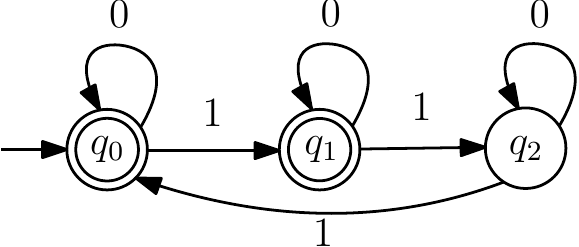}
\caption{Automaton on the alphabet $S=\{0,1\}$. \label{fig:petitautomate}}
\end{figure}

The language $L(\A)$ recognized by the automaton $\A$ is the set of words $\w\in S^*$ such that $\delta(q_0,\w)\in F$.  In the graph, $L(\A)$ corresponds to the set of \emph{successful paths} which are the paths going from $q_0$ to a state $q\in F$. In the example of Figure~\ref{fig:petitautomate}, the automaton recognizes the language of all words on $\{0,1\}$ such that the number of occurrences of $1$ is congruent to $0$ or $1$ modulo $3$.

 Languages which can be recognized by a finite automaton are called \emph{regular}. Regular languages are closed under complementation and intersection. In general they are known to coincide with the class of so-called rational languages; this is a famous theorem of Kleene.

\begin{proposition}
 \label{prop:regulartoseries}
 The length generating function of a regular language is a rational series.
 \end{proposition}
\begin{proof}  Let us sketch a proof of this well-known fact. Given a regular language $L(\A)$, consider the matrix $A=(a_{qq'})_{q,q'\in S}$  where  $a_{qq'}$ is the number of oriented edges from $q$ to $q'$ in the graphic representation of $\A$ (\emph{i.\,e.} the number of $s\in S$ such that $\delta(q,s)=q'$). Then the entry $p_{qq'}$ in the matrix $P:=I+tA+t^2A^2+\cdots=(I-tA)^{-1}$ is the length generating functions of all paths from $q$ to $q'$. Thanks to the adjugate matrix formula, this entry is a rational function in $t$ with $\det(I-tA)$. To obtain the length generating function of $L(\A)$, we notice that it is the sum of $p_{q_0,q}$ for  $q\in F$, so it is also a rational function with denominator $\det(I-tA)$.
\end{proof}

  As a consequence of the proof, notice that the degree of the denominator of the rational function is not larger than the number of states of $\A$.

\subsection{Reduced expressions of FC elements} 

For the rest of this chapter, we fix a Coxeter system $(W,S)$ with graph $\Gamma=\Gamma(W,S)$. In this paragraph, we shall define an automaton $\Ared$ who recognizes the set of reduced words of FC elements in $(W,S)$.

  The main question for designing such an automaton is the following: given such a word $\w=s_1\cdots s_k$, what information is enough to ensure that $\w s$ is a reduced FC word ? First it is clear that this only depends on the element $w$ represented by $\w$. Lemma~\ref{lem:FCcommclasses} tells us that we should consider the commutation class of $\w s$ and check that it contains no square $s^2$ or braid word. In terms of heaps, we need to check that $\H(\w s)=\H(\w)\cdot s$ is a FC heap, i.e. the conditions form Proposition~\ref{prop:bij_fc_elements_heaps} hold.
 
  Before stating the main definition, notice that if $w$ is a FC element and $\w'$ is any word, then we can define without ambiguity $\C(w\w')$ as the commutation class $\C(\w\w')$ where $\w$ is any reduced expression for $w$. The heap $\H(w\w')$ is defined similarly.

\begin{definition}
\label{defi:qe}
Let $w$ be a FC element with heap $H$, and $s,t\in S$ be such that $3\leq m_{st}<\infty$. We define $q_{s,t}(w)=q_{s,t}(H)$ to be the word $\braidl{ts}{r}\in S^*$ which is maximal among alternating suffixes of the chain $H_{st}$ such that $\C(\w\braidr{ts}{m_{st}-r})$ contains an element with a factor $\braidr{st}{m_{st}}$ or $\braidr{ts}{m_{st}}$. 
\end{definition}

Equivalently, $r$ is chosen maximal so that $\H(\w\braidr{ts}{m_{st}-r})$ contains a convex, alternating $st$-chain of length $m_{st}$.

It is important to note that $q_{s,t}(w)$ need not occur as a suffix of any reduced word for $w$. Indeed, consider the FC element $w=s_1s_2s_3$ in type $A_3$, i.e. with braid relations $s_1s_3=s_3s_1, s_1s_2s_1=s_2s_1s_2, s_2s_3s_2=s_3s_2s_3$. Then $q_{s_1,s_2}(w)=s_1s_2$ as is easily checked.
\medskip

 Define $q(H)=(q_e(H))_e$ where $e$ runs through all edges with finite label.  We represent $q(H)$ by a diagram as follows: we start from $\Gamma(W,S)$ and keep only its edge labels equal to $+\infty$. For all other edges $e$, if $q_e(H)=\braidr{st}{r}$ with $r>0$, then we orient $e$ towards $t$ and label it with $r$; if $q_e(H)=\epsilon$ we leave the edge unlabeled and unoriented.

\begin{example}
\label{ex:w0state}
 Consider the Coxeter graph $\Gamma(W_0,\{s,t,u\})$ from Example~\ref{ex:coxetergraph}, and the FC element $stsuts$ whose heap is $H_2$ on the right of Figure~\ref{fig:heaps}. Then $q_{s,t}(H_2)=sts$ and $q_{t,u}(H_2)=\epsilon$, which we represent by the diagram \raisebox{-0.3
cm}{\mbox{\includegraphics{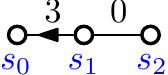}}}.
\end{example}

 Our main automaton on the alphabet $S$ is defined in the following. The important cases of the definition of the transition function are illustrated in Figure~\ref{fig:transitions}, where we use diagrams to represent $(q_e)_e$ as explained for $q(H)$ above.

 \begin{description}
 \item[States]  There is first a {\em sink state} denoted by $\dagger$. All other states are of the form $\left(M,(q_e)_e\right)$ where $M\subseteq S$,  $e$ runs through all pairs $\{s,t\}$ such that $3\leq m_{st}<\infty$ and  $q_e$ is an alternating word in $s$ and $t$ of length $<m_{st}$.
 \item[Initial state] It is given by $Q_0:=\left(\emptyset,(\epsilon)_e\right)$.
 \item[Final States] All states except the sink $\dagger$  are final.
 \item[Transition function] For each state $q$ and each letter $s$, we define a state $q'=\delta(q,s)$ as follows:
\end{description}  
 
\begin{enumerate}
\item In the following three cases, we set $q':=\dagger$:
\begin{enumerate} 
\item $q=\dagger$; \label{it:fail1}
\item $q=\left(M,(q_e)_e\right)$ and $s\in M$; \label{it:fail2}
\item $q=\left(M,(q_e)_e\right)$ and there exists $e=\{s,t\}$ with $3\leq m_{st}<+\infty$ and $q_e=\braidl{st}{m_{st}-1}$.\label{it:fail3}
\end{enumerate}
\item Otherwise we set $q':=\left(M',(q'_e)_e\right)$, with $M':=\{s\}\cup\{t\in M~;st=ts\}$ and $q'_e$ is determined as follows:\label{it:max}
\begin{enumerate}
\item If no endpoint of $e$ is adjacent to $s$, then $q'_e:=q_e$.\label{it:up1}
\item Suppose $s\in e$, say $e=\{s,t\}$.
\begin{enumerate}
\item If $q_e=\braidl{st}{k}$, $k\geq 0$, then $q'_e:=\braidl{sts}{k+1}$.\label{it:up2}
\item If $q_e=[\cdots ts]_{k}$, $k>0$ then $q'_e:=s$.\label{it:up3}
\end{enumerate} 
\item Suppose $s\notin e$ and is adjacent to both endpoints of $e$, then $q'_e:=\epsilon$.\label{it:up4}
\item \label{it:up56} Suppose $s$ is adjacent to exactly one endpoint of $e=\{t,u\}$, say $t$. \begin{enumerate}
\item If $q_e=\braidl{tu}{k}$, $k>0, $ then set $q'_e:=u$ if $u\in M$ while $q'_e:=\epsilon$ if $u\notin M$.\label{it:up5}
\item If $q_e=\braidl{ut}{k}$, $k\geq 0$, then set $q'_e:=\epsilon$ if $k<m_{st}-1$ while $q'_e:=\braidl{ut}{m_{st}-1}$ if $k=m_{st}-1$.\label{it:up6}
\end{enumerate}
\end{enumerate}
\end{enumerate}

A state $q$ is called \emph{accessible} if there is a word $\w$ such that $\delta(q_0,\w)=q$. Obviously the language defined by an automaton is not modified if only its accessible states are kept.

\begin{definition}[Automatn $\Ared(W)$]
\label{defi:Ared}
Given a Coxeter system $(W,S)$, define the automaton $\Ared(W)$ on the alphabet $S$ to be the the automaton defined above restricted to its accessible states.
\end{definition}

\begin{figure}
\includegraphics[width=\textwidth]{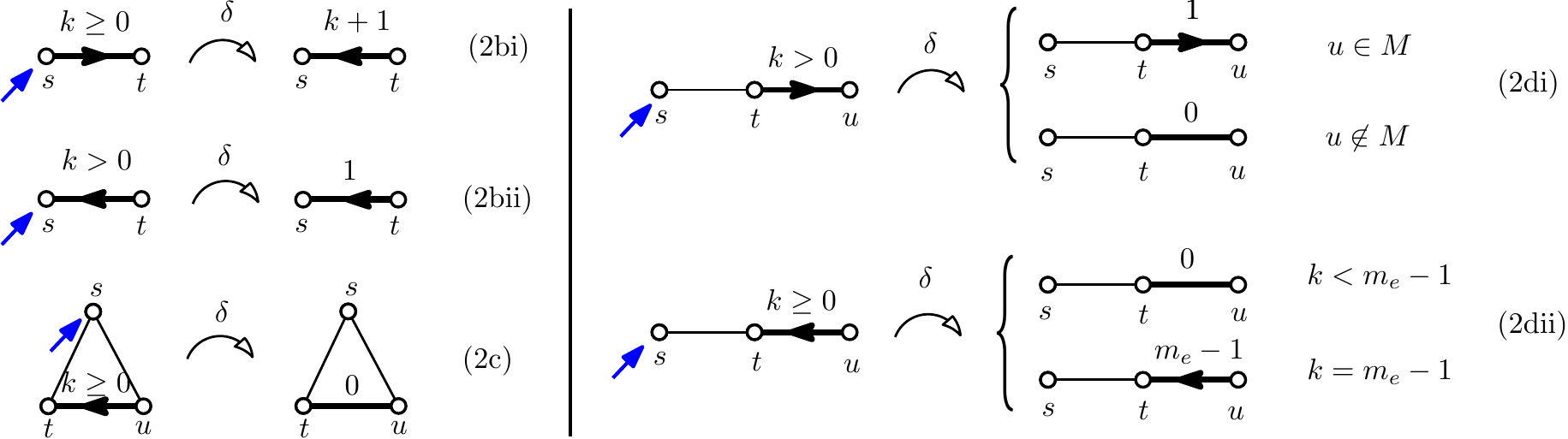}
\caption{The action of $\delta$. The oriented edge is the one being modified, while the letter $s$ to be read is depicted with a small blue arrow.
\label{fig:transitions}}
\end{figure}

For instance, the automaton recognizing $\R(W_0^{FC})$ is illustrated in Figure~\ref{fig:AutRedW0}. It has $19$ states; the sink state is not represented in the picture.

\begin{figure}[!ht]
\includegraphics[width=0.9\textwidth]{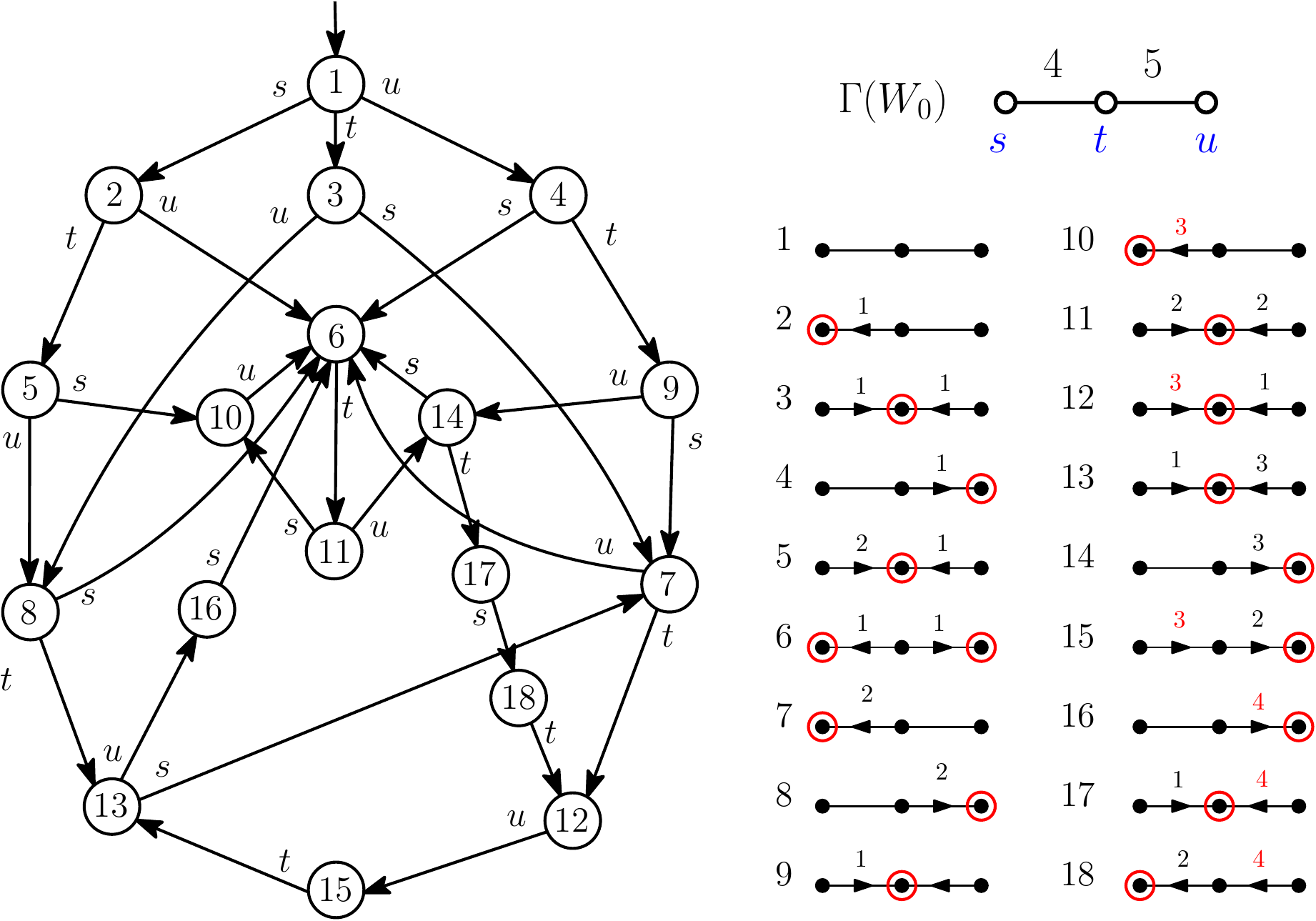}
\caption{Automaton $\Ared(W_0)$ recognizing the language $\R(W_0^{FC})$. All states are final (the sink is not represented.)
\label{fig:AutRedW0}}
\end{figure}

Let us mention two special cases of the automata $\Ared$. 

If $\Gamma(W,S)$ has no finite edges, then $\R(\Wfc)$ consists of words that have no squares $s^2$ in elements of their commutation class. It is easily seen that the knowledge of maxima of the heaps of $\R(\Wfc)$ is enough to recognize the language, which is precisely what $\Ared(W)$ does in this case. 

If $\Gamma(W,S)$ is a complete graph, so that no commutations are possible, then elements of $\Wfc$ are precisely the elements of $W$ which have only one reduced expression. It is then easy to show that $\Ared(W)$ recognizes $\R(\Wfc)$: note in particular that transitions of type~\eqref{it:up56} never occur. In fact it is true in general that for any $(W,S)$, the reduced words for those elements of $W$ with a unique reduced expression form a regular language~\cite{cassweb}.
\medskip

The core of this section is the following lemma:

\begin{lemma}
\label{lem:technical}
 Let $\w$ be a word in $S^*$ and denote by $H$ its heap $\H(\w)$. \\
 $\bullet$  If $H$ is not a FC heap, then $\delta(Q_0,\w)=\dagger$. \\
 $\bullet$  If $H$ is a FC heap, then $\delta(Q_0,\w)=(M,(q_e)_e)$ where $M=\max(H)$ and $q_e=q_e(H)$ for any edge $e$ with $m_e<+\infty$. 
\end{lemma}

We postpone the proof until Section~\ref{sub:prooflemma}, and state the main result of this section.

\begin{theorem}
\label{th:reducedfcregular}
Let $(W,S)$ be a Coxeter system. Then the language $\R(\Wfc)$ is recognized by $\Ared(W)$.
\end{theorem}

\begin{proof}
By Lemma~\ref{lem:technical}, a word $\w$ is recognized by the automaton $\Ared$ if and only if its heap is a FC heap. But this is the case precisely if $\w$ is a reduced expression for a FC element, which completes the proof.
\end{proof}

Before we show how to go from reduced expressions to elements, we make two remarks.

\begin{remark}
\label{rem:canaut}
Recall that the complete set of reduced words $\R(W)$ is also known to be regular, see~\cite[Theorem 4.8.3]{BjorBrenbook}. The automaton $\A_{sm}(W)$ recognizing it uses the standard geometric representation of $W$ and properties of the associated roots: the key is that the set of so-called {small roots} is finite, and the states of $\A_{sm}(W)$ are subsets of small roots. 
Now in all the cases we computed, it seems that $\A_{sm}(W)$ is a larger automaton than $\Ared(W)$. It is then natural to wonder if some sub-automaton of $\A_{sm}(W)$ actually recognize the subset $\R(\Wfc)$ of $\R(W)$. More precisely, does there exist a subset of the edges of $\A_{sm}(W)$ such that the successful paths avoiding these edges correspond precisely to the elements of $\R(\Wfc)$?
\end{remark}

\begin{remark}
\label{rem:minimality} Automata theory tells us there exists a unique automaton with the smallest number of states recognizing a regular language $L$, the {\em minimal automaton} of $L$. This question is investigated in~\cite{HNW} for the language $\R(W)$ and the automaton $\A_{sm}(W)$ mentioned in Remark~\ref{rem:canaut}.

We can ask the same question regarding $\R(\Wfc)$. Some limited computer experimentations (cf. Remark~\ref{rem:implementation}) tend to show that $\Ared(W)$ is always close to being minimal, and minimal in numerous cases.
\end{remark}

\subsection{From words to elements}
\label{sub:wordstoelements}
We showed that reduced words of FC elements form a regular language. We are now interested in FC elements of $W$ themselves, and our strategy will be to pick for each $w\in \Wfc$ a reduced word of $w$ (thus having $\ell(w)$ letters) in order to obtain a regular language.

 For the rest of this section we fix a total order $<$ of $S$, which induces the lexicographical order on $S^*$: $\w_1<\w_2$ if either $\w_1$ is a strict prefix of $\w_2$ or $\w_1=\w s\w'$, $\w_2=\w t\w"$ and $s<t$ in $S$. For each $w\in W$ we let $\min(w)$ be the lexicographically smallest element of $\R(w)$. 
 
 We finally define $\mathcal{L}(W,S,<)=\{\min(w)\,:\,w\in \Wfc\}\subseteq S^*$.

\begin{theorem}
\label{th:Wfcregular}
For any Coxeter system $(W,S)$ and any total order $<$ on $S$, the language $\mathcal{L}(W,S,<)$ is regular.
\end{theorem}

Theorem~\ref{th:rationality} is an immediate corollary by using Proposition~\ref{prop:regulartoseries} and the fact that $\mathcal{L}(W,S,<)$ has generating function $\Wfc(t)$.

\begin{proof}
   Associate to each $\Gamma$-commutation class $\C\subseteq S^*$ its lexicographically minimal element $\min(C)$, and let $Shortlex(\Gamma)$ be the language of all $\min(\C)$. We claim that $\mathcal{L}(W,S,<)=Shortlex(\Gamma)\cap \R(\Wfc)$. Indeed $\mathcal{L}(W,S,<)$ is clearly included in this intersection. Now take a reduced word $\w$ in $\R(\Wfc)$, so it represents an element $w\in\Wfc$ for which $\R(w)$ is a $\Gamma$-commutation class. If $\w\in Shortlex(\Gamma)$, then $\w=\min(\R(w))=\min(w)$ and thus is in $\mathcal{L}(W,S,<)$, which proves the claim.
 
 Now by a result of Anisimov and Knuth~\cite{AnisimovKnuth}, the language $Shortlex(\Gamma)$ can be represented by a rational expression; by the celebrated theorem of Kleene, this means that $Shortlex(\Gamma)$ is a regular language. Since $\R(\Wfc)$ is regular by Theorem~\ref{th:reducedfcregular}, we have written $\mathcal{L}(W,S,<)$ as the intersection of two regular languages, and so it is itself regular.
\end{proof}

\begin{remark}
In the previous proof, one can replace $Shortlex(\Gamma)$ by the language $\operatorname{MinRed}$ obtained by picking for each $w\in W$ the lexicographically smallest reduced expression of $w$. Indeed it was proved in~\cite{BrinkHowlett} that $\operatorname{MinRed}$ is a regular language, and clearly $\mathcal{L}(W,S,<)=\operatorname{MinRed}\cap \R(\Wfc)$. However the proof that it is regular is much more complicated than for $Shortlex(\Gamma)$.
\end{remark}

\begin{remark}
\label{rem:implementation}
We were able to fully implement the computation of $\Wfc(t)$: given a Coxeter graph $\Gamma(W,S)$, our program outputs a rational function $P(t)/Q(t)$ representing $\Wfc(t)$. It was done in \textsc{Sage}~\cite{sage} using the module on automata.

 The implementation follows closely the proof of Theorem~\ref{th:rationality}. First, one constructs the automaton $\Ared(W)$ recognizing $\R(\Wfc)$, by using the definition of $\delta$ to construct the states and transitions inductively. Then, one constructs the automaton recognizing $Shortlex(\Gamma)$, via a nondeterministic automaton recognizing $S^*\setminus Shortlex(\Gamma)$ built using the result of~\cite{AnisimovKnuth}. Then the product automaton recognizing $\mathcal{L}$ is constructed, and from this one can compute the desired rational function as sketched in the proof of Proposition~\ref{prop:regulartoseries}.
\end{remark}

\subsection{Proof of Lemma~\ref{lem:technical}}
\label{sub:prooflemma}

The proof goes by induction on the length of $\w$. It is obviously true for $\w=\epsilon$ by definition of $Q_0$.
\smallskip

Now we assume the claim holds for a word $\w$, and we will show it for  $\w':=\w s$ where $s$ is any letter. We shall denote by $H$ and $H'$ the heaps of the words $\w$ and $\w'$ respectively.
\smallskip

$\bullet$ Assume first that $H'$ is not a FC heap. If $H$ is not FC either, then by induction $\delta(q_0,\w)=\dagger$ so that $\delta(q_0,\w s)=\delta(\dagger,s)=\dagger$ by ~\eqref{it:fail1} as wanted.

 If $H$ is FC, then we use Proposition~\ref{prop:bij_fc_elements_heaps}: $H$ satisfies conditions (h1) and (h2) while $H'$ must fail at least one of the two. If $H'$ fails (h1), then necessarily  $s\in \max(H)$ since $H$ itself satisfies (h1). Now by induction $M=\max(H)$ so we obtain $\delta(q_0,\w s)=\delta((M,(q_e)_e),s)=\dagger$ by~\eqref{it:fail2}. If $H'$ fails (h2), then by~\cite[Lemma 4.1]{St1}, there exists $t$ in $S$ with $3\leq m_{st}<+\infty$ such that $\braidl{st}{m_{st}-1}$ is a convex chain in $H$ and $\braidl{st}{m_{st}}$ is a convex chain in $H'$; in particular $q_{st}(H)=\braidl{st}{m_{st}-1}$. By induction $q_{st}=\braidl{st}{m_{st}-1}$ and thus
  $\delta(q_0,\w s)=\delta((M,(q_e)_e),s)=\dagger$ by~\eqref{it:fail3}.
\smallskip

$\bullet$ Assume now that $H'$ is FC. $H$ is necessarily FC, and by induction $\delta(Q_0,\w)=(M,(q_e)_e)$ where $M=\max(H)$ and $q_e=q_e(H)$ for any edge $e$. 

We need to prove that $\left(M',(q'_e)_e\right):=\delta((M,(q_e)_e),s)$ is such that $M'=\max(H')$ and $q'_e=q_e(H')$ for any edge $e$ with finite label.\smallskip

 By the definition of heaps, $H'$ is obtained from $H$ by adding an element $x_s$ labeled $s$, relations $v<x_s$ if the label $t$ of $v$ satisfies $m_{st}\geq 3$, and then take the transitive closure. Therefore $\max(H')$ is composed of $s$ and the elements of $\max(H)$ such that $m_{st}=2$, which is exactly how $M'$ is defined in~\eqref{it:max}.

Now let $e$ be an edge of $\G$ with $m_e<+\infty$. We distinguish the same sub-cases as we did in Case~\eqref{it:max} in the definition of $\delta$:
\begin{enumerate}[label=(\alph*)]
\item If no endpoint of $e$ is adjacent to $s$, then $x_s$ has no influence on the convexity of the chain $H_e$ and so $q_e(H')=q_e(H)$.
\item We now assume $s\in e$, say $e=\{s,t\}$.
\begin{enumerate}[label=(\roman*)]
\item If $q_e(H)=\braidl{st}{k}$, then adding $x_s$ to $H$ will by definition extend the corresponding alternating convex chain, and so $q_e(H')=\braidl{sts}{k+1}$.\item If $q_e(H)=\braidl{ts}{k}$, then clearly $q_e(H')=s$ corresponding to the $1$-element chain $x_s$.
\end{enumerate}
\item If $s$ is adjacent to both endpoints of $e$, then the alternating convex chain determined by $q_e(H)$ cannot be extended in $H'$, because the new element prevents convexity; therefore $q_e(H')=\epsilon$.
\item The remaining case is when $s$ is adjacent to one endpoint of $e$, say $e=\{t,u\}$ and $s$ is adjacent to $t$ and not to $u$.
\begin{enumerate}[label=(\roman*)]
\item If $q_e(H)=\braidl{tu}{k}$ then, when $k>1$, the addition of $x_s$ will prevent the extension of $q_e(H)$, since any later addition of $t$ will make the corresponding chain non convex. From this we get that if $u\in\max(H)$, then $u\in\max(H')$  so $q_e(H')=u$; if $u\in\max(H)$ then $q_e(H')=\epsilon$
\item If $q_e(H)=\braidl{tu}{k}$, then the corresponding convex alternating chain in $H'$ can be extended by one $t$-element. However it cannot be extended by two elements (labeled $t$ and $u$) because $x_s$ then prevents convexity. So if $k=m_e-1$ we get $q_e(H')=q_e(H)$, while if $k<m_e-1$ we are reduced to $q_e(H')=\epsilon$.
\end{enumerate}
\end{enumerate}

By inspection, we have indeed $q_e(H')=q'_e$ in all cases, which completes the induction step, and the proof of Lemma~\ref{lem:technical}.

\section{Ultimately periodic families of FC elements}
\label{sec:classification}

In this section we  classify Coxeter groups $W$ whose associated FC length counting sequence $\left(\left|\Wfc_l\right|\right)_{l\geq 0}$ is ultimately periodic; we call \emph{FC-periodic} such a group.
 This is a natural extension of Stembridge's result~\cite{St1} who classified \emph{FC-finite} Coxeter groups, \emph{i.\,e.} with a {\em finite} number of FC elements: the result is the family (PF), illustrated in Figure~\ref{fig:FC_Periodic_Graphs}.

We can focus on irreducible groups, thanks to the following lemma.

\begin{lemma}
Let $W=W_1\times \cdots\times  W_k$ be written as a product of irreducible Coxeter groups. Then $W$ is FC-periodic iff all factors $W_i$ are FC-finite except possibly one of them which is FC-periodic.
\end{lemma}

The proof follows easily from the easy decomposition $W^{FC}=W_1^{FC}\times \cdots \times W_k^{FC}$. Thanks to this lemma, we focus on irreducible groups from now on.

\subsection{Classification}
Consider the family $\mathcal{P}:=(PF)\cup (PA)\cup (PE)$ of irreducible Coxeter groups given by the union of the following three families, see Figure~\ref{fig:FC_Periodic_Graphs}:
\begin{enumerate}
\item[(PF)] \label{fam1} Irreducible FC-finite Coxeter groups;
\item[(PA)] \label{fam2} Irreducible affine Coxeter groups different from $\widetilde{E}_8$ and $\widetilde{F}_4$;
\item[(PE)] \label{fam3} Exceptional groups of types $Z_3$ and $Z_6$.
\end{enumerate}

The affine types $\widetilde{E}_8$ and $\widetilde{F}_4$ are excluded from (PA) since they already occur in (PE) under the names $E_9$ and $F_5$. Note that the index for a family in (PF) gives the number of generators of $W$, while in (PA) it corresponds to this number minus one.

\begin{figure}[!ht]
\includegraphics[width=\textwidth]{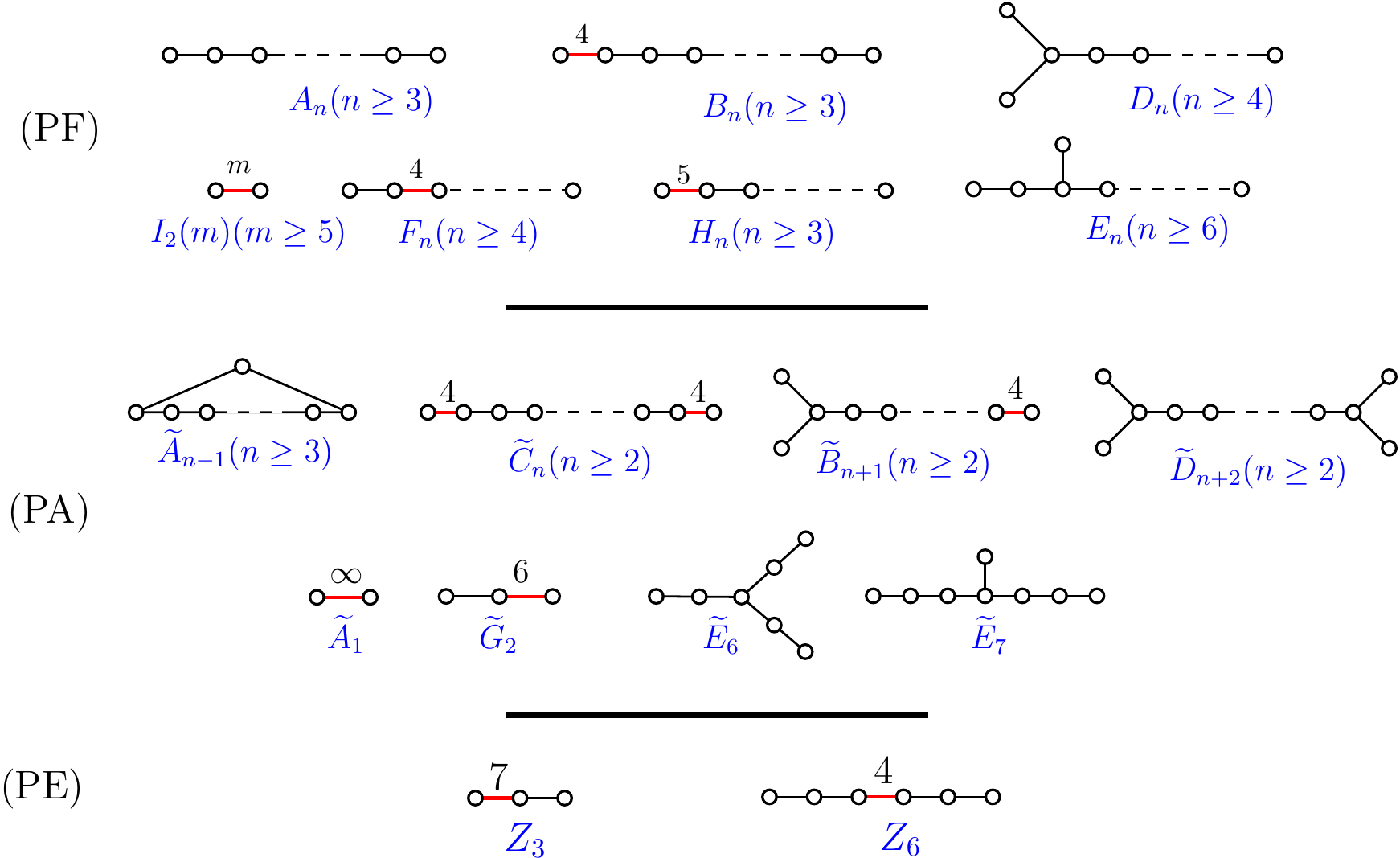}
\caption{The Coxeter diagrams of the family $\mathcal{P}$.
\label{fig:FC_Periodic_Graphs}}
\end{figure}

We can now restate Theorem~\ref{th:classification_intro}.

\begin{theorem}[Classification]
\label{th:classification}
 Let $W$ be an irreducible Coxeter group. Then $W$ is FC-periodic if and only if $W$ has type in $\mathcal{P}$. If $W$ is not FC-periodic, then $(|\Wfc_l|)_l$ grows exponentially fast.
\end{theorem}

One implication in this theorem follows from the following result.

\begin{proposition}
\label{prop:periodic_families}
 If $W$ has type in $\mathcal{P}$, then $W$ is FC-periodic.
\end{proposition}

\begin{proof} This is obvious for elements of (PF), since the sequence is ultimately constant equal to $0$. The fact that all elements of (PA) are FC-periodic is the main results of \cite{BJN}. 

For the two remaining types $Z_3$ and $Z_6$ of (PE), there are two ways to show FC-periodicity. The first one, which we only sketch, is to do as for exceptional affine Coxeter groups in~\cite[Section 5.2]{BJN}: one must examine the structure of the corresponding FC heaps closely, and observe that a certain pattern repeats itself periodically when the size is large enough.  The case of $Z_3$ can be treated as $\widetilde{G}_2$ in~\cite{BJN}, while the case of ${Z}_6$ is slightly more involved.

     The second way is to take advantage of the results of Section~\ref{sec:rationality}, which allows us to exhibit the rational generating function $\Wfc(t)$, see Remark~\ref{rem:implementation}. For $Z_3$, the generating function has the form $P(q)/(1-q^5)$ and so we get a period $5$ when the length is large enough. In the case of $Z_6$, one gets a denominator reduced to $(1-q)$: the growth function is ultimately constant, and is in fact equal to $108$ for any length $l>20$. 
\end{proof}
   

 It remains to prove that for any type not in $\mathcal{P}$, the growth sequence increases exponentially fast. To do this we first obtain a technical result which will allow us to focus on certain minimal types.

Given two Coxeter graphs $G,H$, write $G\leq H$ if $G$ is isomorphic (as an unlabelled graph) to a subgraph $K$ of $H$ (not necessarily induced) such that the edge labels of $G$ are less than or equal to those of $K$.

\begin{lemma}
\label{lemma:monotony}
Let $(W_1,S_1),(W_2,S_2)$ be irreducible Coxeter groups with graphs $\G_1,\G_2$ satisfying $\G_1\leq \G_2$. Then $|\Wfc_{1,l}|\leq |\Wfc_{2,l}|$ for all $l\geq 0$.
\end{lemma}

\begin{proof} 
By hypothesis, we have an injective morphism from $\G_1$ to $\G_2$ so we can safely assume that $S_1\subseteq S_2$ and that the morphism is the identity on $\G_1$. Notice that $\G_1\leq \G_2$ then means precisely that $m^{(1)}_{st}\leq m^{(2)}_{st}$ for all $s,t\in S_1$.

Let $w$ be any element in $\Wfc_{1,l}$. Its set of reduced words $\C:=\R(w)$ is a $W_1$-commutation class of words on $S_1$ of length $l$. By hypothesis $m^{(2)}_{st}=2$ implies $m^{(1)}_{st}=2$ for any $s,t\in S_1$, which shows that $\C$ is a union of $W_2$-commutation classes.

Pick one such $W_2$-commutation class $\C'\subseteq \C$, and consider $\w\in \C'$. Then $\w$ has no squares $s^2$, since $\w\in \C$ and thus is a reduced word for $(W_1,S_1)$. Now consider $s,t\in S_1$ such that $3\leq m^{(2)}_{st}<+\infty$. If $3\leq m^{(1)}_{st}\leq m^{(2)}_{st}$, then no braid word $\braidr{sts}{m^{(1)}_{st}}$ occurs in $\w$ and so \emph{a fortiori} no braid word $\braidr{sts}{m^{(2)}_{st}}$ occurs either. If $m^{(1)}_{st}=2$, then no braid factor $\braidr{sts}{m^{(2)}_{st}}$ can occur in $\w$ since it would then fail to be reduced for $W_1$. 

We showed that $\C'$ is a FC commutation class for $W_2$, and so corresponds to an element in $w'\in \Wfc_{2,l}$ by Lemma~\ref{lem:FCcommclasses}. The map $w\mapsto w'$ is therefore an injection from $\Wfc_{1,l}$ to $\Wfc_{2,l}$.\end{proof}


\subsection{Proof of Theorem~\ref{th:classification}}
Let $G$ be a Coxeter graph which is not in any of the families (PF), (PA) or (PE). We will first determine all such graphs which are minimal for the order $\leq$ introduced above. Denote by $\mathcal{M}$ the family of Coxeter graphs depicted in Figure~\ref{fig:minimal_graphs}.

\begin{figure}[!ht]
\includegraphics[width=0.9\textwidth]{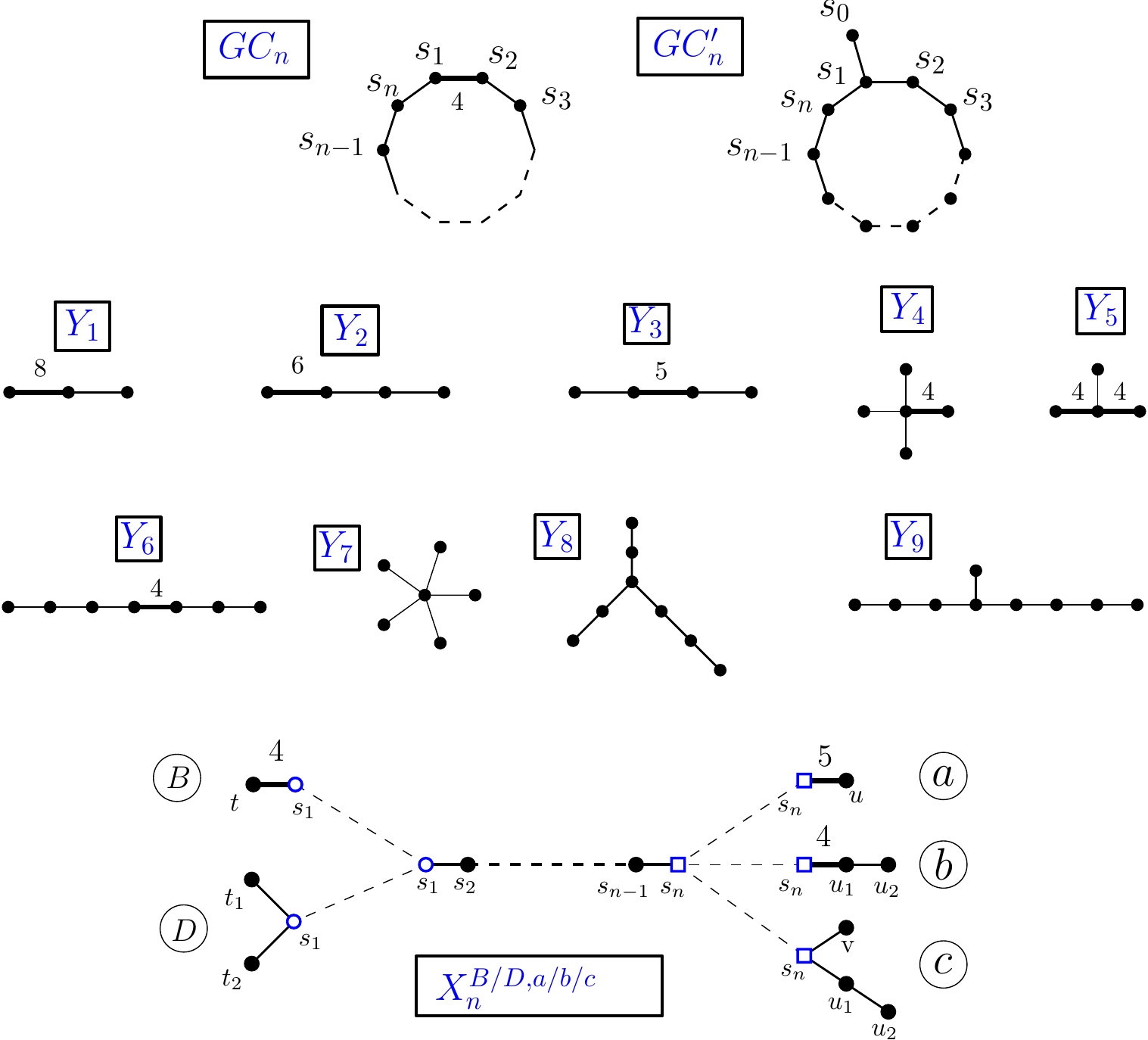}
\caption{\label{fig:minimal_graphs} The family $\mathcal{M}$ of Coxeter graphs. Types $GC_n$ and $GC'_n$ are defined for $n\geq 3$, and types $X_n^{*,*}$ for $n\geq 1$. 
}
\end{figure}

\begin{lemma}
\label{lemma:mingraphs}
If $G$ is a connected Coxeter graph not in $\mathcal{P}$, then there exists $G'\in \mathcal{M}$ such that $G'\preceq G$.
\end{lemma}

\begin{proof}
{\em Case $1$: $G$ contains a cycle}. Pick a minimal cycle in $G$. Since $G\neq\widetilde{A}_n$, one has $GC_n\preceq G$ or $GC'_n\preceq G$ for a certain $n\geq 3$, and we are done in this case.
\smallskip

From now on, we assume that $G$ has no cycle, i.e. $G$ is a tree.
\smallskip

{\em Case $2$: $G$ contains an edge label $m\geq 5$}. Let $e=(u,v)$ be an edge with label $m$. Since $G\neq I_2(m)$, $G$ has at least three vertices.  If $G$ has exactly three vertices, then since $G\neq Z_3,H_3$, we have $X_1^{B,e}\preceq G$ or $Y_1\preceq G$ . If neither $u$ nor $v$ is a leaf, then $Y_3\preceq G$. If $u$ or $v$ has degree $\geq 3$, then $X_1^{D,a}\preceq G$. 

There remains the case where $G$ has vertices $x,y$ with edges $(v,x)$ and $(x,y)$. If $m\geq 6$, then $Y_2\preceq G$. If $m=5$, $G\neq H_n$ implies $X_n^{B,a}\preceq G$ or $X_n^{D,a} \preceq G$ for a certain $n$.
\smallskip

{\em Case $3$: $G$ has maximal edge label $4$.} Assume first that the label $4$ occurs at least twice, and pick two edges labeled $4$ with simple edges on the path linking them. Since $G\neq \widetilde{C}_n$, there is an edge in $G$ not on this path; depending on the position of this edge one gets $X_n^{B,i}\preceq G$ for $i=b$ or $c$, $Y_4\preceq G$.

Now suppose $G$ has exactly one edge $e=(u,v)$ labeled $4$.  If $G$ has no branch point, then since $G\neq B_n, F_n, Z_6$, one has necessarily $Y_6\preceq G$. If $G$ has a branch point of degree $d\geq 4$, one has $Y_4\preceq G$ or $X_n^{D,c}\preceq G$ for a certain $n\geq 1$; if the branch point has degree $3$, then since $G\neq\widetilde{B}_n$ one has $X_n^{B,c}\preceq G$ or $X_n^{D,b}\preceq G$  for a certain $n\geq 1$.
\smallskip

{\em Case $4$: $G$ has all edge labels equal to $3$.} Since $G\neq A_n$, $G$ has necessarily at least one branch point, \emph{i.\,e.} a vertex of degree larger than $2$. If $G$ has two branch points or more, then since $G\neq \widetilde{D}_n$ one has $X_n^{D,c}\preceq G$ for an $n\geq 1$. 

  We may now assume $G$ has exactly one branch point, say of degree $d$. If  $d\geq 5$, then $Y_7\preceq G$. If  $d=4$, then since $G\neq \widetilde{D}_4$ one has $X_1^{D,c}\preceq G$. If $d=3$, then since $G\neq E_n,\widetilde{E}_6,\widetilde{E}_7$ one has necessarily $Y_8\preceq G$ or $Y_9\preceq G$.

This completes the proof.\end{proof}

\begin{lemma}
\label{lemma:unboundedgraphs}
 For each Coxeter group $W$ with type in $\mathcal{M}$, the sequence $(|\Wfc_{l}|)_{l\geq 0}$ grows exponentially.
\end{lemma}

\begin{proof}
Let us first deal with the exceptional cases $Y_i$, $i=1,\ldots,9$. It is possible in each case to prove the non-boundedness directly, by coming up with specific FC elements or heaps. For instance, in the case of $Y_2$, $w_1=s_1s_2s_1s_3s_2$ and $w_2=s_1s_2s_1s_3s_2s_1s_4s_3s_2$ are FC elements, and one checks that all possible products of $w_1$ and $w_2$ are reduced and form distinct elements of $\Wfc$.

We rather proceed as in Proposition~\ref{prop:periodic_families}, by using the results of Section~\ref{sec:rationality} and our implementation mentioned in Remark~\ref{rem:implementation}. This gives us the rational generating functions $Y_i^{FC}(t)$, whose denominators we recorded in Figure~\ref{tb:minina}. In each case, there is a positive root smaller than $1$, which forces an exponential growth for $Y_i^{FC}$.
\begin{figure}[!ht]
\[\begin{array}{|c|c|c|} 
   \hline
     \text{Type} & \text{Denominator of } \Wfc(t) & \text{Smallest positive root}   \\
    \hline
    Y_1 & 1-q^5-q^7 & 0.8898912458 \\
    Y_2  & 1-q^5 - q^9 - q^{11} & 0.8699890716 \\
    Y_3  & 1- q^4 - q^6 & 0.8688369618 \\
    Y_4  & 1- 2q^3- 3q^5- 2q^6 -q^7 + 6q^8 + 2q^9+ 4q^{12}+ q^{13} & 0.6950969040 \\
    Y_5  & 1- q^2 - q^3 - 2q^6 - q^7 + 3q^8 + q^9 + q^{12} & 0.9225155924 \\
    Y_6  & 1-q^{12}-q^{17}-q^{25} & 0.9378483025 \\
    Y_7  & 1- 3q^3 -q^7  & 0.6778224161 \\
    Y_8  & 1-q^{12}-q^{17}-q^{25} & 0.9378483025 \\
    Y_9  & 1-q^{18}-q^{37}& 0.9740122556 \\
    \hline
\end{array}
\]
\caption{For each group $W$ of type $Y_i$ illustrated in Figure~\ref{fig:minimal_graphs}, we compute the denominator of the rational function \label{tb:minina}}
\end{figure}
\smallskip

 In type $GC_n$, define the words $\w=s_3s_4\cdots s_ns_1s_2$ and $\w'=s_ns_{n-1}\cdots s_2s_1$. Now consider the set $E$ of words of the form  $u_1\epsilon_1 u_2\ldots\epsilon_{k-1}u_k$, where  $u_i\in\{\w,\w'\}$, $\epsilon_i=\epsilon$ if $u_i=u_{i+1}$, $\epsilon_i=s_1$ if $u_iu_{i+1}=\w\w'$ while $\epsilon_i=s_2$ if $u_iu_{i+1}=\w'\w$. Notice that these words in $E$ have no adjacent commuting letters, and no braid factor. $E$ can be naturally considered as a subset of $\Wfc$ which grows exponentially, and this proves the claim in type $GC_n$. 
 
 In type $GC'_n$, let $\w$ be the word $s_1s_2\cdots s_ns_1s_2\cdots s_n$. Now consider the set $E'$ of words $\w\epsilon_1\w\ldots\epsilon_{k-1}\w$, where $\epsilon_i$ is either $\epsilon$ or $s_0$. These are all reduced words for distinct FC elements (notice that the only possible commutations involve the letters $s_0$). So $E'$ clearly corresponds to a family in $\Wfc$ with exponential growth.
\smallskip

 Finally, one notices that types $X_n^{*,*}$ resemble types $\widetilde{B},\widetilde{C},\widetilde{D}$. In~\cite{BJN}, it was shown that these were FC-finite, and in fact the corresponding FC heaps were entirely classified. In particular there were two infinite families,  alternating and \emph{zigzag} heaps. Zigzag heaps are given by words occurring as factors of $(tx_1x_2\cdots x_nux_n\cdots x_1)^\infty$ in type $\widetilde{C}$.
 
  We now fix a zigzag element $w$ of type $\widetilde{C}$ with heap $H$. We describe how to obtain a family of FC elements in type $X_n^{B,i}$ for $i=a,b$ and $c$:

\begin{enumerate}[label=(\alph*)]
\item Insert $u$-elements such that, in the word $H_{x_n,u}$, these occurrences are always separated by at least one $x_n$.
\item Insert $u_2$-elements so that in the word $H_{u_1,u_2}$, these occurrences are always separated by at least two $u_1$'s.
\item Change all occurrences of $u$ by $u_1v$ to get a word with heap $H'$. Then insert $u_2$-elements so that in the word $H'_{u_1,u_2}$, these occurrences are always separated by at least two $u_1$.
\end{enumerate} 

In each case, one verifies that the heaps created are indeed FC heaps. Also, when the initial zigzag element increases in length, more and more heaps are created. This shows that in each type the heaps created determine a family that grows exponentially with length. The case $X_n^{D,i}$ are dealt with in the same manner by replacing occurrences of $t$ by $t_1t_2$ in the resulting heaps.\end{proof}


We can now conclude the proof of Theorem~\ref{th:classification}. By Proposition~\ref{prop:periodic_families},  all graphs in $\mathcal{P}$ are FC-periodic. Let $\G$ be a Coxeter graph which is not in $\mathcal{P}$. By Lemma~\ref{lemma:mingraphs}, there exists $\G'\in \mathcal{M}$ with $\G'\preceq \G$. The counting sequence of FC elements of $\G'$ is exponential by Lemma~\ref{lemma:unboundedgraphs}, so by Lemma~\ref{lemma:monotony} the counting sequence of FC elements of $\G$ also is, which completes the proof.

\section{Further work}
\label{sec:perspectives}
  We now give some possible continuations for this work; some of them were mentioned in the course of the text.
\medskip

There are several questions of interest regarding our automaton $\Ared(W)$ recognizing $\R(\Wfc)$. For various families of Coxeter groups, for instance finite and affine types, one could look at the combinatorics of the states of $\Ared(W)$, by studying their structure, enumerate them, etc. On a related note, and as mentioned in Remark~\ref{rem:minimality}, it would be  interesting to study questions of minimality. 

Also, $\Ared(W)$  is defined in an \emph{ad hoc} manner requiring in particular a technical disjunction of cases. In Remark~\ref{rem:canaut}, we raised the question of designing another automaton recognizing $\R(\Wfc)$ based on the pre-existing  automaton $\A_{sm}(W)$ recognizing all $\R(W)$.
 \medskip

As for the classification result given in Theorem~\ref{th:classification_intro}, it would be interesting to find a uniform explanation for the periodic behaviour; this question was already asked in~\cite{BJN} for affine types. When $\Wfc$ grows exponentially on the other hand, it is natural to wonder if there is a ``gap'' in the exponential rates of growth: if $\alpha_W$ is defined by $\liminf_l \log(|\Wfc_l|)/l$,  are the $\alpha_W$ bounded away from $1$ when $W$ runs through non FC-periodic irreducible Coxeter groups ? Note that, as a byproduct of our proof, it is enough to study this question for the groups in $\mathcal{M}$, see Figure~\ref{fig:minimal_graphs}. 

\bibliographystyle{plain}

\begin{thebibliography}{10}

\bibitem{AnisimovKnuth}
A.~V. An{\={\i}}s{\={\i}}mov and D.~E. Knuth.
\newblock Inhomogeneous sorting.
\newblock {\em Internat. J. Comput. Inform. Sci.}, 8(4):255--260, 1979.

\bibitem{BenkartMeinel}
G.~{Benkart} and J.~{Meinel}.
\newblock {The center of the affine nilTemperley-Lieb algebra}.
\newblock {\em ArXiv e-prints}, May 2015.

\bibitem{BJN}
R.~Biagioli, F.~Jouhet, and P.~Nadeau.
\newblock Fully commutative elements in finite and affine {C}oxeter groups.
\newblock {\em Monatsh. Math.}, 178(1):1--37, 2015.

\bibitem{BjorBrenbook}
A.~Bj{\"o}rner and F.~Brenti.
\newblock {\em Combinatorics of {C}oxeter groups}, volume 231 of {\em Graduate
  Texts in Mathematics}.
\newblock Springer, New York, 2005.

\bibitem{BrinkHowlett}
B.~Brink and R.~B. Howlett.
\newblock A finiteness property and an automatic structure for {C}oxeter
  groups.
\newblock {\em Math. Ann.}, 296(1):179--190, 1993.

\bibitem{cassweb}
B.~Casselman.
\newblock Coxeter groups part {II}. {W}ord processing.
\newblock CRM Winter School on Coxeter groups, 2002.

\bibitem{deMan}
R.~de~Man.
\newblock The generating function for the number of roots of a {C}oxeter group.
\newblock {\em J. Symbolic Comput.}, 27(6):535--541, 1999.

\bibitem{sage}
The~Sage Developers.
\newblock {\em {S}age {M}athematics {S}oftware ({V}ersion 6.9)}, 2015.
\newblock {\tt http://www.sagemath.org}.

\bibitem{FanStem}
C.~K. Fan and J.~R. Stembridge.
\newblock Nilpotent orbits and commutative elements.
\newblock {\em J. Algebra}, 196(2):490--498, 1997.

\bibitem{Fan}
C.K. Fan.
\newblock {\em A Hecke algebra quotient and properties of commutative elements
  of a {W}eyl group}.
\newblock Phd thesis, M.I.T., 1995.

\bibitem{FominGreene}
S.~Fomin and C.~Greene.
\newblock Noncommutative {S}chur functions and their applications.
\newblock {\em Discrete Math.}, 193(1-3):179--200, 1998.
\newblock Selected papers in honor of Adriano Garsia (Taormina, 1994).

\bibitem{Graham}
J.~Graham.
\newblock {\em Modular Representations of Hecke Algebra s and Related
  Algebras}.
\newblock PhD thesis, University of Sydney, 1995.

\bibitem{GreenBook}
R.~M. Green.
\newblock {\em Combinatorics of Minuscule Representations}.
\newblock Cambridge Tracts in Mathematics. Cambridge University Press, 2013.

\bibitem{HanJon}
C.~R.~H. Hanusa and B.~C. Jones.
\newblock The enumeration of fully commutative affine permutations.
\newblock {\em European J. Combin.}, 31(5):1342--1359, 2010.

\bibitem{HNW}
C.~{Hohlweg}, P.~{Nadeau}, and N.~{Williams}.
\newblock {Automata, reduced words, and Garside shadows in Coxeter groups}.
\newblock {\em ArXiv e-prints}, October 2015.

\bibitem{Humphreys}
J.~E. Humphreys.
\newblock {\em Reflection groups and {C}oxeter groups}, volume~29 of {\em
  Cambridge Studies in Advanced Mathematics}.
\newblock Cambridge University Press, Cambridge, 1990.

\bibitem{JonesAnnals}
V.~F.~R. Jones.
\newblock Hecke algebra representations of braid groups and link polynomials.
\newblock {\em Ann. of Math. (2)}, 126(2):335--388, 1987.

\bibitem{GrowthAlgebrasBook}
G.~R. Krause and T.~H. Lenagan.
\newblock {\em Growth of algebras and {G}elfand-{K}irillov dimension},
  volume~22 of {\em Graduate Studies in Mathematics}.
\newblock American Mathematical Society, Providence, RI, revised edition, 2000.

\bibitem{Matsumoto}
H.~Matsumoto.
\newblock G\'en\'erateurs et relations des groupes de {W}eyl g\'en\'eralis\'es.
\newblock {\em C. R. Acad. Sci. Paris}, 258:3419--3422, 1964.

\bibitem{Postnikov}
A.~Postnikov.
\newblock Affine approach to quantum {S}chubert calculus.
\newblock {\em Duke Math. J.}, 128(3):473--509, 2005.

\bibitem{Saka}
Jacques Sakarovitch.
\newblock {\em Elements of automata theory}.
\newblock Cambridge University Press, Cambridge, 2009.
\newblock Translated from the 2003 French original by Reuben Thomas.

\bibitem{St1}
J.~R. Stembridge.
\newblock On the fully commutative elements of {C}oxeter groups.
\newblock {\em J. Algebraic Combin.}, 5(4):353--385, 1996.

\bibitem{St2}
J.~R. Stembridge.
\newblock Some combinatorial aspects of reduced words in finite {C}oxeter
  groups.
\newblock {\em Trans. Amer. Math. Soc.}, 349(4):1285--1332, 1997.

\bibitem{TemperleyLieb}
H.~N.~V. Temperley and E.~H. Lieb.
\newblock Relations between the ``percolation'' and ``colouring'' problem and
  other graph-theoretical problems associated with regular planar lattices:
  some exact results for the ``percolation'' problem.
\newblock {\em Proc. Roy. Soc. London Ser. A}, 322(1549):251--280, 1971.

\bibitem{Tits68}
J.~Tits.
\newblock Le probl\`eme des mots dans les groupes de {C}oxeter.
\newblock In {\em Symposia {M}athematica ({INDAM}, {R}ome, 1967/68), {V}ol. 1},
  pages 175--185. Academic Press, London, 1969.

\bibitem{Ufnarovskij}
V.~A. Ufnarovskij.
\newblock Combinatorial and asymptotic methods in algebra.
\newblock In {\em Algebra, {VI}}, volume~57 of {\em Encyclopaedia Math. Sci.},
  pages 1--196. Springer, Berlin, 1995.

\bibitem{ViennotHeaps}
G.~X. Viennot.
\newblock Heaps of pieces. {I}. {B}asic definitions and combinatorial lemmas.
\newblock In {\em Combinatoire \'enum\'erative ({M}ontreal, {Q}ue.,
  1985/{Q}uebec, {Q}ue., 1985)}, volume 1234 of {\em Lecture Notes in Math.},
  pages 321--350. Springer, Berlin, 1986.

\bibitem{Ukraine_four}
M.~V. Zavodovskii.
\newblock Growth of generalized {T}emperley-{L}ieb algebras associated with
  {C}oxeter graphs (four projectors).
\newblock {\em Dopov. Nats. Akad. Nauk Ukr. Mat. Prirodozn. Tekh. Nauki},
  (7):12--15, 2010.

\bibitem{Ukraine_simple}
M.~V. Zavodovskii and Yu.~S. Samoilenko.
\newblock Growth generalized {T}emperley-{L}ieb algebras connected with simple
  graphs.
\newblock {\em Ukrainian Math. J.}, 61(11):1858--1864, 2009.

\end{thebibliography}

\def\lfhook#1{\setbox0=\hbox{#1}{\ooalign{\hidewidth
  \lower1.5ex\hbox{'}\hidewidth\crcr\unhbox0}}}

\end{document}